\documentclass{article}

\pdfoutput=1

\usepackage{amsmath, amssymb, amsthm, enumerate, fullpage}
\usepackage{verbatim, enumitem, bbm, etoolbox}

\newcommand{\m}{\mathcal}
\renewcommand{\b}{\mathbb}

\apptocmd{\lim}{\limits}{}{}

\newcommand{\cl}{\textrm{cl}}
\newcommand{\im}{\textrm{im}}

\newcommand{\tp}{\textrm{tp}}
\renewcommand{\o}{\overline}

\newcommand{\LO}{\textrm{LO}}

\theoremstyle{plain}
\newtheorem{thm}{Theorem}

\newtheorem{lemma}[thm]{Lemma}
\newtheorem{cor}[thm]{Corollary}

\numberwithin{thm}{section}

\numberwithin{subcase}{case}

\theoremstyle{definition}
\newtheorem{definition}[thm]{Definition}
\newtheorem{remark}[thm]{Remark}

\def\Ind{\setbox0=\hbox{$x$}\kern\wd0\hbox to 0pt{\hss$\mid$\hss}
\lower.9\ht0\hbox to 0pt{\hss$\smile$\hss}\kern\wd0}

\def\Notind{\setbox0=\hbox{$x$}\kern\wd0\hbox to 0pt{\mathchardef
\nn=12854\hss$\nn$\kern1.4\wd0\hss}\hbox to 0pt{\hss$\mid$\hss}\lower.9\ht0
\hbox to 0pt{\hss$\smile$\hss}\kern\wd0}

\newcommand{\Mod}{\textrm{Mod}}
\newcommand{\iso}{\cong}

\newenvironment{claim}[1]{\par\noindent\underline{Claim:}\space#1}{}
\newenvironment{claimproof}[1]{\par\noindent\underline{Proof:}\space#1}{\leavevmode\unskip\penalty9999 \hbox{}\nobreak\hfill\quad\hbox{$\square$}}

\begin{document}

\bibliographystyle{plain}
 
\author{Richard Rast\footnote{The first-named author's contribution partially supported by NSF grant DMS-1308546} \\ Davender Singh Sahota\footnote{The second-named author's contribution first appeared in his PhD thesis, University of Illinois at Chicago, 2013}}
\title{The Borel Complexity of Isomorphism for O-Minimal Theories} 
\date{\today} 

\maketitle

\begin{abstract}
	Given a countable o-minimal theory $T$, we characterize the Borel complexity of isomorphism for countable models of $T$ up to two model-theoretic invariants.  If $T$ admits a nonsimple type, then it is shown to be Borel complete by embedding the isomorphism problem for linear orders into the isomorphism problem for models of $T$.  This is done by constructing models with specific linear orders in the tail of the Archimedean ladder of a suitable nonsimple type.
	
	If the theory admits no nonsimple types, then we use Mayer's characterization of isomorphism for such theories to compute invariants for countable models.  If the theory is small, then the invariant is real-valued, and therefore its isomorphism relation is smooth.  If not, the invariant corresponds to a countable set of reals, and therefore the isomorphism relation is Borel equivalent to $F_2$.
	
	Combining these two results, we conclude that $(\Mod(T),\iso)$ is either maximally complicated or maximally uncomplicated (subject to completely general model-theoretic lower bounds based on the number of types and the number of countable models).
\end{abstract}

\nocite{vanDenDriesBook}
\nocite{PillaySteinhorn}
\nocite{hjorthBook}

\section{Introduction}

In 1988, Laura Mayer proved Vaught's Conjecture for o-minimal theories in a surprising way- an o-minimal $T$ either has finitely many countable models or continuum many.  This was accomplished through a sharp dichotomy she introduced: whether or not $T$ admits a ``nonsimple type.''  If $T$ admits a nonsimple type, $T$ must have continuum-many models.  If not, the isomorphism relation can be simply characterized, and $T$ has continuum-many countable models if and only if there are infinitely many nonisolated types.

In this paper we sharpen this divide, completely characterizing where $(\Mod(T),\iso)$ lies in the Borel complexity hierarchy (here and throughout, $\Mod(T)$ will refer to the space of countable models of $T$ with universe $\omega$).  Most prominently we show that, given a nonsimple type, $T$ is \emph{Borel complete} -- as complicated as an isomorphism problem can possibly be, so that every isomorphism problem is effectively reducible to this one.  In particular, this is proved by reducing the isomorphism problem for linear orders into the one for models of $T$; by Theorem~3 in \cite{FriedmanStanleyBC}, every isomorphism problem is reducible to this one, so this is sufficient.

If there is no such type, and there are few types generally (that is, $S_1(T)$ is countable, or equivalently, $T$ is small), then $T$ is \emph{smooth}.  That is, there is a uniformly computable real-valued invariant for countable models which characterizes them up to isomorphism.  In this case, that invariant is just the set of types which are realized on the model, and whether or not the set of realizations has endpoints.  Since smooth relations form the bottom of the hierarchy, such $T$ are Borel reducible to any other equivalence problem which has enough classes to fit the classes of $T$.

In between these two cases, if there is no nonsimple type but $S_1(T)$ is uncountable, then $T$ is not smooth.  However, countable models are still characterized by what types are realized and what kinds of endpoints they have.  Therefore, $(\Mod(T),\iso)$ will be Borel equivalent with $F_2$, meaning each model is classified up to isomorphism by a uniformly computable \emph{countable set} of reals.  It turns out (see Theorem~1.3 in \cite{markerNonSmallTheories}) that this is a lower bound for all non-small theories, and therefore even in this case, $T$ is minimal among non-small theories.

Therefore the divide is as sharp as it could be- $(\Mod(T),\iso)$ is either maximal among all isomorphism problems, or minimal among all problems with which it can be reasonably compared.  Together, all of this comprises our main theorem:

\begin{thm}\label{MainTheorem}
	Let $T$ be a complete o-minimal theory in a countable language.
	\begin{enumerate}
		\item If $T$ has no nonsimple types and $S_1(T)$ is countable, then $(\Mod(T),\iso)$ is smooth.
		\item If $T$ has no nonsimple types and $S_1(T)$ is uncountable, then $(\Mod(T),\iso)$ is Borel equivalent to $(2^\omega,F_2)$.
		\item If $T$ admits a nonsimple type, then $(\Mod(T),\iso)$ is Borel complete.
	\end{enumerate}
	
\end{thm}

Also, although it is not a major concern of this paper, it should be noted that adding finitely many parameters cannot change the isomorphism complexity between these three cases. Even adding infinitely many parameters can cause a move from the smooth case to $F_2$, but no other moves are possible.

Finally, we end with two interesting corollaries to the main theorem.  The first states that any nontrivial o-minimal theory (in particular, any theory which defines an infinite group) is Borel complete.  The second states that any discretely o-minimal theory (or even one with an infinite discrete part) is Borel complete.  Together, these imply virtually all known Borel completeness results for concrete examples of o-minimal theories.

This question was originally explored in the second-named author's PhD thesis \cite{SahotaThesis}, where a partial form of the main result was shown and where many of the techniques of this paper were first employed.  In this thesis, the concept of faithfulness was introduced and shown to be sufficient for Borel completeness.  From there, the general case was reduced to the faithful nonsimple non-cut case by adding constants to the language, demonstrating Borel completeness after adding finitely many parameters.

The first-named author read that work and independently completed the result, building on the techniques of the second-named author's thesis.  He used a refined notion of nonsimplicity to eliminate the use of parameters at several points, as well as to solve the ``cut'' case directly.  Finally, he introduced the idea of a canonical tail to deal with the case of a nonsimple atomic interval, allowing for the complete result.  This work is released as a coauthored paper to recognize the independent, if asynchronous, contributions of both authors.

Both authors would like to particularly credit Laura Mayer, who in her solution to Vaught's Conjecture for o-minimal theories \cite{MayerVC}, defined nonsimplicity and proposed the dichotomy which inspired this work.  Further, her characterization of isomorphism in the ``no non-simple types'' case laid the foundation for almost everything in Section~\ref{NoNonsimpleTypesSection}.

\section{Background and Proof Outline}

Throughout, $T$ will refer to an o-minimal theory in a countable language.  We will \emph{not} assume the underlying order is dense.  However, as we will see in Theorem~\ref{DiscreteBC}, if the order has an infinite non-dense part, then there is a very simple solution to our main question, which shows that $T$ is Borel complete directly.  Where relevant, we will work in an $\aleph_1$-saturated ``monster'' model, from which all parameters will be drawn.  All models will be elementary substructures of this monster model.

\subsection{O-Minimality}

The theory of o-minimality was first developed in \cite{PillaySteinhorn} by Pillay and Steinhorn and in \cite{KnightPillaySteinhorn} by Knight, Pillay, and Steinhorn.  In particular, they proved that o-minimal theories have constructible models over sets, and therefore have prime, atomic models over sets which are unique up to isomorphism.  For any set $A$, refer to this model as $\Pr(A)$.  This will be sufficiently well-defined for our purposes, since we will only depend on the isomorphism type, rather than any specific embedding of it into the monster model.

They also proved a cell-decomposition theorem for definable sets and a continuity-monotonicity theorem for definable functions.  We will use the latter of these two frequently, often without explicit mention.  The following is an easy consequence of the continuity-monotonicity theorem:

\begin{remark}
	Suppose $p$ and $q$ are complete $A$-types, and $f$ is an $A$-definable function which sends some realization of $p$ to some realization of $q$.  Then for any model $\m M$ including $A$,  $f$ is a bijection from $p(\m M)$ to $q(\m M)$ which is either order-preserving or order-reversing.  We will refer to this property by saying ``$f$ is a bijection from $p$ to $q$."
\end{remark}

In \cite{MarkerOTT}, Marker identified the three different kinds of complete nonalgebraic 1-types which can arise.  A complete 1-type $p$ which has both a definable infimum $L$ and a supremum $R$ is atomic, and generated by the atomic interval $(L,R)$; note that $L$ or $R$ may be among $\pm\infty$.  If $p$ has a definable infimum \emph{or} a definable supremum, but not both, then $p$ is nonisolated and is called a ``non-cut.''  Finally, if $p$ has neither endpoint, then $p$ is nonisolated and is called a ``cut.''

Using the remark above, we can easily see that if $p$ and $q$ are complete $A$-types, and $f:p\to q$ is $A$-definable, then $p$ and $q$ are of the same kind: both atomic, both non-cuts, or both cuts.  Further, in \cite{MarkerOTT}, Marker showed that realizing a cut does not force a realization of a non-cut.  More specifically:

\begin{thm}[Marker]
	Let $M$ be a model of an o-minimal theory.  Let $p\in S_1(M)$ be a non-cut, and $q\in S_1(M)$ be a cut.  Let $a\models p$ and $b\models q$.  Then $p$ is not realized in $\Pr(Mb)$ and $q$ is not realized in $\Pr(Ma)$. 
\end{thm}

Concerning closures; we evaluate all closures in the monster model, though it's equivalent to evaluating closure inside any elementary substructure containing the parameter set, and we consider all of our models to be elementarily embedded in the monster model.  Since the theory is ordered, algebraicity is equivalent to definability.  We refer to this closure operation as $\cl$.  It follows from continuity-monotonicity that $\cl$ satisfies \emph{exchange}.  That is, for any elements $a$ and $b$, and for any set $A$, if $a\in\cl(Ab)$ and $a\not\in\cl(A)$, then $b\in\cl(Aa)$.

Since closure will be used quite heavily, we introduce two parameters to it.  For any set $A$, the expression $\cl_A(X)$ means $\cl(A\cup X)$.  More unusually, if $p$ is a 1-type, possibly using parameters, then $\cl^p(X)$ is the set of all $X$-definable points which realize $p$.  These notations can and will be combined: $\cl^p_A(X)$ refers to the set of all $A\cup X$-definable elements which realize $p$.

\subsection{Nonsimplicity}

The next fundamental notion is that of \emph{nonsimplicity}, originally due to Mayer and introduced in \cite{MayerVC}:

\begin{definition}
	A type $p\in S_1(A)$ is \emph{simple} if, for every set $B$ of realizations of $p$, $\cl^p_A(B)$ is $B$.
	
	Say $p$ is \emph{nonsimple} if $p$ is not simple; that is, for some set $B$ of realizations of $p$, there is a $b\not\in B$ which realizes $p$ and which is $B$-definable.
\end{definition}

By compactness, if $p$ is nonsimple, there is a \emph{finite} set $B$ satisfying the above.  In particular, we will say $p$ is $n$-nonsimple if there is some $B$ as above with $|B|\leq n$.  We will say $p$ is $n$-simple if there is no such set $B$.  The following remark makes the \emph{minimal} nonsimplicity index very interesting:

\begin{remark}
	If $p$ is $k$-simple, then the type $p^{k+1}(x_0, \ldots, x_k)$ generated by $\{x_0<\cdots<x_k\}\cup \bigcup_{i=0}^k p(x_i)$ is complete.
\end{remark}

If $p$ is $n$-nonsimple, then there is some ascending $n$-tuple $\o a$ and some element $b$ which is $\o a$-definable but not in $\o a$.  If $n$ is \emph{minimal} such that $p$ is $n$-nonsimple, then by the remark, $p^n$ is a complete type.  Therefore, every $\o a$ has this property, so we get a definable function $f:p^n\to p$ such that $f(\o a)=b$.

While we will be very interested in particular nonsimple types, we are using the existence of a nonsimple type as a property of the theory which forms the most important dividing line for complexity.  Since the use of parameters can be a significant obstacle to descriptive set theoretic analysis, the following lemma is extremely helpful:

\begin{lemma}
	Let $A$ be a finite set, and suppose $p\in S_1(A)$ is $n$-nonsimple.  Then the restriction $p_0$ of $p$ to $S_1(\emptyset)$ is is $n+|A|$-nonsimple.
\end{lemma}
\begin{proof}
	By an obvious inductive argument, we assume $A$ is a singleton $a$, and $p\in S_1(a)$.  Assume $n$ is minimal such that $p$ is $n$-nonsimple.  By way of contradiction, suppose that $p_0$ is $n+1$-simple.  By nonsimplicity and the remark, there is an $a$-definable function $f:p^n\to p$.  In fact, we may take $f$ to be $f(\o x;y)$ such that $f(\o x;a)$ is a nontrivial function $p^n\to p$, and by exchange over $a$, may assume that $f(\o x;a)$ is defined on ascending tuples $x_1<\cdots<x_n$ and satisfies $f(\o x;a)>x_n$ everywhere.
	
	Suppose that $\cl^{p_0}(a)$ is nonempty; that is, there is an $a'\in\cl(a)$ which realizes $p_0$.  But then by exchange, $a\in\cl(a')$, so $f(\o x;a)$ is $a'$-definable, witnessing $n+1$-nonsimplicity of $p_0$.  So it must be that $p$ is equivalent to $p_0$; therefore we take $p=p_0$ for the remainder of the proof.  Let $q(y)=\tp(a)$.
	
	But now the type $p_0^n\times q$ is complete, and $f$ is a function $p_0^n\times q\to p_0$.  By exchange and the preparation above, we may replace $f$ with a function $g:p_0^{n+1}\to q$.  Assume $g$ is of minimal arity with this property.  Then for any $b_1<\cdots<b_n$ from $p_0$, the function $g(\o b; x_{n+1})$ is a bijection from the complete $\o b$-type $p_0(x)\cup\{x>b_n\}$ to the complete $\o b$-type $q$.
	
	Therefore, define the function $h:p_0^{n+1}\to p_0$ by $h(x_1,\ldots, x_{n+1})$ to be the unique $y>x_{n+1}$ from $p_0$ where $g(x_1, \ldots, x_{n+1})=g(x_2, \ldots, x_{n+1}, y)$; such a $y$ must exist and be greater than $x_{n+1}$ by the above proof, yielding $n+1$-nonsimplicity of $p_0$, as desired.
\end{proof}

In fact, since functions require only finitely many parameters to be defined, if $T$ admits a nonsimple type over \emph{any} set, then $T$ has a nonsimple type over the empty set.  Although we will not use it, the converse is also true -- any nonsimple type over any set $A$ admits a nonsimple extension to any $B\supset A$.  Therefore:

\begin{cor}\label{parametersAreUnneeded}
	For a complete o-minimal $T$, the following are equivalent:
	\begin{itemize}
		\item $T$ admits a nonsimple type over $\emptyset$.
		\item $T$ admits a nonsimple type over $A$, for \emph{some} set $A$.
		\item $T$ admits a nonsimple type over $A$, for \emph{every} set $A$.
	\end{itemize}
	
\end{cor}

We will refer to any of the above conditions on $T$ as \emph{admitting a nonsimple type}.

\subsection{Borel Complexity}

We are interested in measuring the complexity of the isomorphism for countable models of $T$, which is a finer measurement than just counting them.  A common way to do this is through the idea of Borel reducibility -- establishing a natural way to see if one relation is ``more difficult'' than another to compute.

The notion of a Borel reduction as a way to compare complexity of classes was introduced by Friedman and Stanley in \cite{FriedmanStanleyBC}.  Consider pairs of the form $(X,E)$, where $X$ is a Borel subset of a Polish space, and $E\subseteq X^2$ is an equivalence relation.  Given two such pairs, we say $(X_1,E_1)\leq_B (X_2,E_2)$ (sometimes written as $E_1\leq_B E_2$) if there is a Borel function $f:X_1\to X_2$ where, for all $a,b\in X$, $aE_1b$ holds if and only if $f(a)E_2f(b)$.  It is clear that $\leq_B$ forms a preorder, and the requirement that $f$ be Borel provides a reasonable analogue to saying ``$E_1$ is effectively computable from $E_2$'' when (as is usual) $X_1$ and $X_2$ are uncountable.

For a countable language $L$, define $\Mod(\omega,L)$ to be the set of $L$-structures with universe $\omega$.  This is made into an uncountable Polish space using the formula topology: for any formula $\phi(\o x)$ and any tuple $\o n$ from $\omega$, the set $\{\m M:\m M\models \phi(\o n)\}$ is open.  The natural equivalence relation on this space is $L$-isomorphism.  It is well-known that any isomorphism-invariant Borel subset $X$ of $\Mod(\omega,L)$ has a corresponding $L_{\omega_1,\omega}$-sentence $\Phi$, such that $X$ is the set of models of $\Phi$ -- see, for example, Theorem 16.8 in \cite{KechrisDST}.

An invariant subset $X$ is called \emph{Borel complete} if \emph{every} invariant space $Y$ is Borel reducible to it; say a sentence $\Phi$ is Borel complete if $(\Mod(\Phi),\iso)$ is.  It is a theorem of Friedman and Stanley that there are many such sentences, including the definitions of graphs, trees, and linear orders.  This last one, whose space we denote $(\LO,\iso)$, will be most important for us. When $T$ admits a nonsimple type, we will reduce $(\LO,\iso)$ to $(\Mod(T),\iso)$ to show Borel completeness of $T$.

We are also interested in the ``less complex'' side of things.  Any $E$ which is Borel (as a subset of $X\times X$) is strictly less complex than the Borel complete relations, and the bottom of this class of Borel equivalence relations is the \emph{smooth} relations $(X,E)$ (sometimes called \emph{tame}).  $(X,E)$ is smooth if it is Borel reducible to $(Y,=)$ for some Polish space $Y$, or equivalently, there if there is a uniformly computable real invariant for elements of $X$ which precisely describes equivalence mod $E$.  Every smooth relation is Borel reducible to every non-smooth relation, and two smooth relations $(X_1,E_1)$ and $(X_2,E_2)$ satisfy $E_1\leq_B E_2$ if and only if $|X_1/E_1|\leq |X_2/E_2|$.

Our final relation of interest is $F_2$, an equivalence relation on real-valued sequences $(X)^\omega$, where $X$ is any uncountable Polish space.  Define $a F_2 b$ if the sets $\{a(0), a(1), a(2), \ldots\}$ and $\{b(0), b(1), b(2), \ldots\}$ are equal.  $F_2$ is not smooth, but it is Borel, and fairly low in the hierarchy.  It is of significant interest to us because of Theorem~1.3 in \cite{markerNonSmallTheories}, which says that if $T$ is countable but not small, then $F_2\leq_B \iso_T$.  That is, $F_2$ is a lower bound on Borel complexity for non-small theories.

\subsection{Proof Overview}

Most of the content of the paper is in proving Theorem~\ref{MainTheorem}.

First, we consider the case where $T$ has no nonsimple types.  We state and re-prove Mayer's characterization of isomorphism for such $T$.  Then we go on to prove the exact place in the Borel hierarchy for $T$ by giving explicit Borel reductions into the appropriate spaces.

The more complicated case is when $T$ has a nonsimple type.  In all cases, we will give a Borel reduction from $(\LO,\iso)$ into $(\Mod(T),\iso)$.  In essence, we will give a Borel function $\LO\to \Mod(T)$ where $L$ appears as the Archimedean ladder of some nonsimple type in $\m M_L$.  In actuality, this only works in the presence of a \emph{faithful} nonsimple type.  The notion of faithfulness applies in different ways depending on the `kind' of nonsimple type we have, so we divide into cases based on whether our nonsimple type is isolated or not.

If $p$ is a nonsimple, nonisolated type, then either $p$ is a non-cut or a cut.  We show that all non-cuts are faithful, and that every cut is either faithful or can be used to produce a nonsimple non-cut (which is necessarily faithful).  Finally, we show how to produce a Borel reduction $(\LO,\iso)$ into $(\Mod(T),\iso)$ given a faithful nonsimple type, completing the proof in this case.

The remaining case is when $p$ is isolated, where there may be no faithful types anywhere.  We exploit the idea that we can add parameters to produce a non-cut, so that we can embed a linear order as the ladder of this type.  This will not be preserved under isomorphism of models, but we show that such an embedding has a \emph{canonical tail} which is preserved.  Finally, we show that this is enough -- there is a Borel complete subclass of $\LO$ where tail isomorphism is equivalent to actual isomorphism, so we can still produce a Borel reduction $(\LO,\iso)\to(\Mod(T),\iso)$.

Therefore, given a nonsimple type, $T$ must be Borel complete.  We end with two corollaries which provide sufficient conditions for Borel completeness.  First, if the underlying order is not almost dense (that is, there are infinitely many non-dense points), we will be able to generate a faithful nonsimple type over $\emptyset$, just using the successor function.  Second, if the theory itself is nontrivial (regardless of whether this happens in a single type), we can use exchange to produce a nonsimple type over finitely many parameters.  In light of the main theorem, this shows Borel completeness for $T$ in either case.  

\section{No Nonsimple Types}\label{NoNonsimpleTypesSection}

The aim of this section is to completely characterize the complexity of $(\Mod(T),\iso)$, in the case that $T$ does not admit a nonsimple type.  Therefore, {\bf for the rest of this section, $T$ is a countable o-minimal theory with no nonsimple types.}  Our characterization will depend entirely on the size of $S_1(T)$ and the number of countable models.  To do this, consider the following definition, which is implicit in \cite{MayerVC}:

\begin{definition}
	Let $\m M$ and $\m N$ be countable models of $T$.  We say that $\m M$ and $\m N$ are \emph{apparently isomorphic} if, for every $p\in S_1(\emptyset)$, $p(\m M)\cong p(\m N)$ as linear orders.
\end{definition}

Our characterization relies on two major facts; that ``apparent'' isomorphism is equivalent to ``actual'' isomorphism, and that apparent isomorphism is a relatively simple thing to compute.

\subsection{Apparent Isomorphism is Equivalent to Isomorphism}

We begin by summarizing the part of Mayer's work which is relevant to us.  All results and definitions in this subsection are due to her and proved in \cite{MayerVC}, although the exposition is new.

\begin{lemma}\label{apparentIsomorphismLemma}
	Given countable models $\m M$ and $\m N$ of $T$, $\m M\cong\m N$ if and only if $\m M$ and $\m N$ are apparently isomorphic.
\end{lemma}

This lemma follows from a back-and-forth argument, using the following lemma as an inductive step:

\begin{lemma}[Mayer]
	Suppose $\m M$ and $\m N$ are countable models of $T$ and $A$ is a finite set of parameters in $M\cap N$ where, for all $p\in S_1(A)$, $p(\m M)\cong p(\m N)$ as linear orders.  Then, for any $a\in \m M$, there is a $b\in\m N$ such that $\tp(a)=\tp(b)$ and for all $q(x;y)\in S_2(A)$, $q(\m M;a)\cong q(\m N;b)$ as linear orders.
\end{lemma}
\begin{proof}
	Let $\m M$ and $\m N$ be as described; clearly we may assume $A=\emptyset$ by adding it into the language.  Let $a\in\m M$ be arbitrary, and for every $p\in S_1(\emptyset)$, let $f_p:p(\m M)\to p(\m N)$ be an order isomorphism as guaranteed by hypothesis.
	
	First, note that any $\emptyset$-definable function between complete 1-types must be a continuous, strictly monotone bijection, either order-preserving or order-reversing; this follows from the continuity-monotonicity theorem and the fact that both types are complete.  Next, note that since all types are simple, there is at most one $\emptyset$-definable function between any two 1-types, since if $f,g:p\to q$ are distinct, then $g^{-1}\circ f:p\to p$ makes $p$ nonsimple.  As a consequence, for every type $p$, there is at most one element $a'\in\cl(a)$ which realizes $p$.  The same holds for any $b$ in $\m N$.
	
	With all this said, fix $p=\tp(a)$ and let $b=f_p(a)\in\m N$.  Observe that $\tp(a)=\tp(b)=p$.  We argue that this choice of $b$ works; that for any $q(x;y)\in S_2(\emptyset)$, $q(\m M;a)\cong q(\m N;b)$.  By the previous paragraph, every type over $\emptyset$ either stays the same or splits into two convex pieces.  If $q(x;a)$ is equivalent to its restriction $q_0$ to $\emptyset$, then so is $q(x;b)$, and they are already isomorphic under $f_q$.  If it does split, then there is an $a'\in\cl(a)$ which realizes $q_0$, so there is a unique $\emptyset$-definable homeomorphism $f:p\to q_0$ where $f(a)=a'$.  Observe that $f$ works in $\m N$ as well, and $q(x;b)$ splits into two pieces over $b'=f(b)$.
	
	Assume that $f:p\to q_0$ is strictly decreasing (the strictly increasing case is similar).  Then $f$ is a strictly decreasing bijection $p\cup\{x>a\}\to q_0\cup\{x<a'\}$ and $p\cup\{x<a\}\to q_0\cup\{x>a'\}$, and similarly for $b$ and $b'$ in $\m N$.  So $f\circ f_p\circ f^{-1}$ is an order-preserving bijection $q_0\cup\{x<a'\}\to q_0\cup\{x<b'\}$ and $q_0\cup\{x>a'\}\to q_0\cup\{x>b'\}$.  Since $q(x;a)$ is either $q_0\cup\{x<a'\}$ or $q_0\cup\{x>a'\}$, and $q(x;b)$ similarly, the function $f\circ f_p\circ f^{-1}$ is the desired order-isomorphism $q(\m M;a)\to q(\m N;b)$, completing the proof.
\end{proof}

It only remains to prove that ``apparent isomorphism'' is a comparatively simple notion to compute.  To that end, consider the following lemma:

\begin{lemma}\label{classifyingIsomorphismTypesLemma}[Mayer]
	For any simple $p\in S_1(\emptyset)$ and any countable $\m M\models T$, if $a,b\in p(\m M)$ and $a<b$, then there is a $c\in p(\m M)$ with $a<c<b$.
	
	Therefore, $p(\m M)$ is order-isomorphic to one of six countable linear orders.
\end{lemma}
\begin{proof}
	First, suppose $p(\m M)$ has at least two elements, $a<b$.  If $a$ has an immediate successor, then at most $b$, and by convexity of $p$, \emph{every} realization of $p$ has an immediate successor which realizes $p$.  Therefore, the successor function is a $\emptyset$-definable witness to $p$ being 1-nonsimple (against hypothesis).  Thus this cannot happen, so no element of $p(\m M)$ has an immediate successor or predecessor.
	
	Therefore, either $|p(\m M)|\leq 1$ (yielding two possible isomorphism types) or $p(\m M)$ is a dense linear order.  In this case, two choices remain with respect to endpoints, and therefore four more possible options for the isomorphism type of $p(\m M)$, for a total of six.
\end{proof}

As a consequence, two models $\m M$ and $\m N$ are apparently isomorphic if and only if, for every type $p\in S_1(\emptyset)$ which is realized in $\m M$ or $\m N$, if $p(\m M)$ has a first element (or last element, or sole element), then so does $p(\m N)$, and vice-versa.

\subsection{The Complexity of Isomorphism for Theories With No Nonsimple Types}

With both models in hand, computing apparent isomorphism is not especially difficult.  Determining precisely how difficult leads to the following characterization:

\begin{thm}\label{simpleComplexityTheorem}
	Suppose $T$ is o-minimal with no nonsimple types.
	
	\begin{itemize}
		\item If there are $c$ pairwise-independent cuts and $n$ pairwise-independent non-cuts over $\emptyset$, both finite, then $(\Mod(T), \cong)$ is Borel equivalent to $(3^n6^c,=)$.
		
		\item If there are an infinite but countable number of pairwise-independent nonisolated 1-types over $\emptyset$, then $(\Mod(T), \cong)$ is Borel equivalent to $(\b R,=)$.
		
		\item If $S_1(T)$ is uncountable, then $(\Mod(T), \cong)$ is Borel equivalent to $((2^\omega)^{\omega},F_2)$.
	\end{itemize}
	
\end{thm}

For the first and second points, we need only show that if $T$ is small, then $T$ is smooth.  In this case the Borel equivalence class of $T$ is defined exactly by the number of nonisomorphic countable models of $T$, a count which has already been done in \cite{MayerVC}, where the notion of ``pairwise-independent'' is also made precise.

\begin{lemma}
	If $S_1(T)$ is countable and has no nonsimple types, then $T$ is smooth.
\end{lemma}
\begin{proof}
	We need a function $F:\Mod(\omega,T)\to X$, for some Polish space $X$, where $\m M_1\iso\m M_2$ iff $F(\m M_1)=F(\m M_2)$.  So let $X=6^{S_1(T)}$, which is a countable product of a finite set and is therefore Polish.  Fix an enumeration of the six possible countable dense linear orders, and for any $\m M\models T$ and $p\in S_1(T)$, let $F(\m M)(p)$ be the index of the order-type of $p(\m M)$.  This function is clearly Borel and satisfies the requirements.
\end{proof}

For the third point, it's enough to show that if $S_1(T)$ is uncountable, then $T\leq_B F_2$.  For as has already been mentioned, $F_2$ embeds into $(\Mod(T),\iso)$ for any non-small theory.

\begin{lemma}
	If $S_1(T)$ is uncountable, then $T\leq_B F_2$.
\end{lemma}
\begin{proof}
	Since $T$ is not small, $X=S_1(T)\times 6$ is an uncountable Polish space.  We will produce a Borel function $F:\Mod(\omega,T)\to X^\omega$ such that $\m M_1\iso\m M_2$ iff $\{F(\m M_1)_n:n\in\omega\}$ and $\{F(\m M_2)_n:n\in\omega\}$ are equal as sets.
	
	To that end, fix an enumeration of the six possible countable dense linear orders, and define $F(\m M)(n)$ be $(\tp_{\m M}(n),k)$, where $k$ is the index of the isomorphism type of $\tp_{\m M}(n)(\m M)$.  This function is again Borel, and two models yield the same set of sequence values if and only if they are apparently isomorphic.
\end{proof}

This proves the main theorem, as well as the following unexpected corollary, which gives another way in which these theories are dominated by their 1-types:

\begin{cor}
	Let $T$ be o-minimal with no nonsimple types.  $T$ is small if and only if $S_1(T)$ is countable.
\end{cor}
\begin{proof}
	If $T$ is small, then $S(T)$ is countable, so $S_1(T)\subseteq S(T)$ is countable.  If $T$ is not small, then $F_2\leq_B T$, so $T$ is not smooth, so $S_1(T)$ is uncountable.
\end{proof}

\section{A Nonsimple Type}\label{nonsimpleTypeSection}

Our goal for this section is to show that if $T$ is a countable o-minimal theory which admits a nonsimple type, then the isomorphism relation for $\Mod(\omega,T)$ is Borel complete.  Therefore, {\bf for the rest of this section, $T$ is a countable o-minimal theory which admits a nonsimple type.}  Since the isomorphism relation on linear orders is known to be Borel complete, our goal will be to show a Borel reduction from $(\LO,\iso)$ to $(\Mod(\omega,T),\iso)$.

Given a complete type $p$ over some set $A$, and for any set $B\supset A$, define an Archimedean equivalence relation on realizations of $p$ as follows: given $a$ and $b$ realizing $p$, say $a\sim_B b$ if there are $a_1,a_2\in \cl^p_B(a)$ and $b_1,b_2\in \cl^p_B(b)$ such that $a_1\leq b\leq a_2$ and $b_1\leq a\leq b_2$.  For our purposes, $A$ will usually be $\emptyset$.  In the quite common case that $A=B=\emptyset$, we will omit the subscript on $\sim$.

This is easily seen to be an equivalence relation.  Moreover, the equivalence classes are \emph{convex}, and thus they are totally ordered.  As a result, given any model $\m M$ of $T$ (containing $B$), the quotient $p(\m M)/\sim_B$ is a linear order, which we call the \emph{Archimedean ladder.}  If $A=B=\emptyset$, this is an invariant of the model which is preserved under isomorphism.  Assuming we can construct models with arbitrary countable ladders, and do this in a Borel fashion, we can give a Borel reduction $(\LO,\iso)\leq_B (\Mod(\omega,T),\iso)$ and show Borel completeness.

The next step toward this is the notion of faithfulness:

\begin{definition}
	A nonsimple type $p\in S_1(A)$ is \emph{faithful} if, for any set $B$ of realizations of $p$, and any $c\in\cl_A^p(B)$, $c\sim_A b$ for some $b\in B$.
\end{definition}

Approximately, ``faithfulness'' says that given some realizations of $p$, you can't access anything too fundamentally different.  In particular, you can't access any new Archimedean classes.  Since o-minimal theories have constructible\footnote{For our purposes, ``constructible'' means ``prime and mechanically producible; otherwise see for instance \cite{MarkerMT}.''} models over sets, this gives a technique: given a countable linear order $L$, pick a \emph{faithful} type $p\in S_1(\emptyset)$ and a set of $\sim$-inequivalent constants which realize $p$, and which are indexed and ordered by $L$.  The constructible model over this set of constants will have ladder exactly isomorphic to $L$, and we're done.  The details will be shown later, but there is no hidden difficulty.  The problem is finding a faithful type at all.

Our first stage is to show that if there is a \emph{nonisolated} nonsimple type over $\emptyset$, then there is a \emph{faithful} nonsimple type over $\emptyset$.  It turns out all nonsimple non-cuts are 1-nonsimple, and all 1-nonsimple non-cuts are faithful, so if there is a nonsimple non-cut, there is a faithful type.  Neither of these properties are true for cuts, but if there is an \emph{unfaithful} cut over $\emptyset$, then we can use it to produce a nonsimple non-cut over $\emptyset$.  So if $T$ admits any nonisolated nonsimple type, we can produce a faithful type and conclude that $T$ is Borel complete.

This does not completely resolve the question, however.  Consider the theory of ordered, affinized divisible abelian groups.  In this case, the only 1-type over $\emptyset$ is given by the atomic formula $x=x$, which is 1-simple but 2-nonsimple.  No such type can be faithful, so this theory admits no faithful types over $\emptyset$.  However, if we add two parameters (call them $0$ and $1$), there is a resulting non-cut ``at infinity,'' which is faithful by the work above.  We can build a ladder in this non-cut by faithfulness, but because the definition of the type relies on parameters, it will not be preserved under isomorphism.  To deal with this, we introduce the notion of a canonical tail:

\begin{definition}
	Let $p\in S_1(T)$ be an atomic nonsimple type, and let $n$ be minimal where $p$ is $n$-nonsimple.  Say $p$ \emph{has a canonical tail} if, for all sets $A$ and $B$ from $p$ of size $n$, $\sim_A$ and $\sim_B$ coincide above $\cl^p(AB)$.  That is, for all elements $c,d$ from $p$, if $c,d>\cl^p(AB)$, then $c\sim_{A}d$ if and only if $c\sim_{B}d$.
\end{definition}

The problem from before is that if we use parameters to construct a ladder, it will not be preserved under isomorphism.  However, if the atomic type has a canonical tail, then any isomorphism between suitably chosen models will preserve a \emph{tail} of the intended linear order.  With this in mind, we will first show that every nonsimple atomic type has a canonical tail.  Next, we will construct a Borel complete class of linear orders on which isomorphism and sharing a tail are the same notion, and use this to show that an atomic type with a canonical tail provides Borel completeness.

Combining these two ideas shows that if $T$ is a countable, o-minimal theory which admits a nonsimple type, then $T$ is Borel complete.

\subsection{A Nonsimple Nonisolated Type}

Our goal in this section is to show that if $T$ admits a nonsimple nonisolated type over $\emptyset$, then $T$ admits a faithful nonisolated type over $\emptyset$.  There are two distinct cases -- non-cuts and cuts.  Before we prove our needed results, we will need one lemma which is used frequently and without explicit mention.  Note that in this lemma and all that follow, we can also work with types over parameters with no change in the argument.

\begin{lemma}\label{bigNMakesClosureDense}
	If $p\in S_1(T)$ is $n$-nonsimple, then for any set $B$ of realizations of $p$ with $|B|\geq n$, $\cl^p_A(B)$ has no first or last element.  Further, if $p$ is 1-simple, then $\cl_A^p(B)$ is a dense linear order.
\end{lemma}
\begin{proof}
	Let $n$ be minimal where $p$ is $n$-nonsimple, and let $\o a=a_1<\cdots<a_n$ be realizations of $p$.  By $n$-nonsimplicity, there is a $b\in\cl^p_{B}(\o a)$ which is not in $\o a$.  By exchange, any of the elements of $\o ab$ is definable over the other $n$.  By minimality of $n$, this yields functions $f_i:p^n\to p$ for $i=0,1,\ldots,n$ where $f_0(\o a)<a_1<f_1(\o a)<\cdots<a_n<f_n(\o a)$.
	
	In particular, $f_0(\o a)<a_1$ and $a_n<f_n(\o a)$, so no sufficiently large set's closure has a first or last element.  If $n\geq 2$, then we can use $f_1$ to get between $a_1$ and $a_2$, establishing density.
\end{proof}

\begin{lemma}
	If $p\in S_1(T)$ is a nonsimple non-cut, then $p$ is $1$-nonsimple.
\end{lemma}
\begin{proof}
	Let $n$ be minimal such that $p$ is $n$-nonsimple.  By exchange and minimality of $n$, construct $f(\o x):p^n\to p$ such that if $\o x=x_1<\cdots<x_n$ are all from $p$, then $x_n<f(\o x)<L$.  Then $f$ is defined and has this property on a convex set below $L$, so there is a $b\in\cl(\emptyset)$ where if $b<x_1<\cdots<x_n<L$, then $f(\o x)$ is defined and $x_n<f(\o x)<L$.
	
	But since $p$ is a non-cut, $\cl(\emptyset)$ approaches $L$ from the left, so there are elements $a_1<\cdots<a_{n-1}$ from $\cl(\emptyset)$ satisfying $b<a_1<\cdots<a_{n-1}<L$.  So the function $g(x)=f(a_1, \ldots, a_{n-1}, x)$ is a nonsimple function from $p$ to $p$, establishing $1$-nonsimplicity.
\end{proof}

\begin{lemma}\label{noncutsAreFaithful}
	If $p\in S_1(T)$ is a nonsimple non-cut, then $p$ is faithful.
\end{lemma}
\begin{proof}
	Suppose $p$ is unfaithful.  We may assume $p$ has a supremum; the infimum case is symmetric.  By unfaithfulness, there is a tuple $a_1<\cdots<a_n$ of realizations of $p$ where $[a_1]<\cdots<[a_n]$ and where there is a $b\in\cl^p(\o a)$ such that $b\not\sim a_i$ for any $i=1,\ldots,n$.  We assume $n$ is minimal with this property.  Clearly $n>1$, so by exchange, we may assume $b<a_1$.  Let $A=\{a_1, \ldots, a_{n-1}\}$.
	
	Observe that $\dim(\o ab)=n$ by minimality of $n$, so that every point of $\o ab$ is definable over the other $n$.  Also, observe that $\tp(b/Aa_n)$ is a cut ($p$ is nonsimple, so the closure of a nonempty set has no first or last element), while $\tp(a_n/Ab)$ is a non-cut.  Both are nonisolated over $A$, so neither is realized in $\Pr(A)$.  Yet $\tp(a_n/A)$ is realized in $\Pr(Ab)$, violating the omitting types theorem in \cite{MarkerOTT}.
\end{proof}

This resolves the issue of non-cuts; we will now prove that unfaithful cuts can be used to produce nonsimple (and therefore faithful) non-cuts.

\begin{lemma}
	If a cut $p\in S_1(T)$ is nonsimple, then it is 2-nonsimple.
\end{lemma}
\begin{proof}
	Suppose not.  That is, let $p\in S_1(T)$ be a nonsimple cut, and let $n$ be minimal such that $p$ is $n$-nonsimple, and such that $n\geq 3$.
	
	By exchange and minimality of $n$, there is a $\emptyset$-definable function $f:p^n\to p$, defined on ascending $n$-tuples from $p$, such that if $x_1<\cdots<x_n$ are realizations of $p$, then $x_1<f(\o x)<x_2<\cdots<x_n$.  Then these properties hold on a convex set, so hold on a $\emptyset$-definable open interval $I=(a,b)$ containing $p$.  Since $\cl(\emptyset)$ approaches $p$ from the right, we can choose elements $c_3<\cdots<c_n$ from $\cl(\emptyset)$ where $c_n<b$ and $p(x)$ implies $x<c_3$.
	
	Then the function $g(x_1,x_2)=f(x_1,x_2,c_3,\ldots,c_n)$ is $\emptyset$-definable and defined on the interval $(a,c_3)$.  Further, if $x_1<x_2$ realize $p$, then $x_1<g(x_1,x_2)<x_2$, so by convexity of $p$, $g(x_1,x_2)$ realizes $p$ as well, establishing 2-nonsimplicity of $p$.
\end{proof}

If our cut has no unary functions, then the binary function can be made to act like an averaging function, defined on every pair in a convex set including the type itself.  Since our type is nonisolated, the domain of this function will overspill the type, and we can use this to construct a nonsimple non-cut:

\begin{lemma}
	If $p\in S_1(T)$ is a nonsimple, 1-simple cut, then there is a faithful non-cut $q\in S_1(T)$.
\end{lemma}
\begin{proof}
	Let $n$ be minimal such that $p$ is $n$-nonsimple.  By hypothesis and the previous lemma, we may conclude that $n=2$.  By exchange, there is a function $g:p^2\to p$ where if $x<y$ realize $p$, then $x<g(x,y)<y$.  But then this property holds on an $\emptyset$-definable interval $I$ containing $p$.  Since $\cl(\emptyset)$ approaches $p$ from the right, let $c\in\cl(\emptyset)$ be some element such that $c\in I$ and $p(x)$ implies $x<c$.  Then the function $g(x,c)$ is $\emptyset$-definable, defined on the open interval $(a,c)$ which includes $p$, and if $a<x<c$, then $x<g(x,c)<c$.
	
	Since $\cl(\emptyset)$ approaches $p$ from the right, let $c'\in\cl(\emptyset)$ be such that $p(x)<c'<c$.  Then the set $\{y\in\cl(\emptyset):a<y<c\}$ is nonempty, and because $g(x,c)$ is $\emptyset$-definable, it must therefore approach $c$ from the left.  So the type $q(x)=\{x<c\}\cup\{x>y:y\in\cl(\emptyset)\land y<c\}$ is a non-cut, and is nonsimple under the function $g(x;c)$.  By Lemma~\ref{noncutsAreFaithful}, $q$ is faithful, completing the proof.
\end{proof}

While the same idea applies to the 1-nonsimple case, the proof is more delicate.  We will temporarily require the notion of \emph{$n$-unfaithfulness}; the property of a type which says that it is unfaithful, and there is a witness of length at most $n$.

\begin{lemma}\label{firstUnfaithfulCutLemma}
	If a cut $p\in S_1(T)$ is 1-nonsimple but 2-unfaithful, then for any $b$ realizing $p$, the non-cut below $b$ is 1-nonsimple as a $b$-type.
\end{lemma}\begin{proof}
Suppose that $p$ is 2-unfaithful, and pick a witnessing pair.  That is, there is some pair $[a]<[b]$ from $p$ and a $\emptyset$-definable $g$ where $[a]<[g(x,y)]<[b]$.  We will show that the non-cut $(b)^-$ is 1-nonsimple as a $b$-type.

Since the type $r(x;y)=p(x)\cup p(y)\cup\{[x]<[y]\}$ is complete, $g$ witnesses unfaithfulness for every sufficiently spread pair.  In particular, $g$ is defined, continuous, strictly increasing in both of its arguments, and satisfies $x<g(x,y)<y$ everywhere on this type.  We may therefore pass to a definable 2-cell $\m U$ containing the descending pair $(b,a)$ on which these properties are satisfied.  Its underlying interval must contain all of $p(x)$, and its boundary functions $L(x)$ and $R(x)$ are everywhere defined on $p(x)$.

Then $L(b)<a<R(b)$; since $[a]<[b]$ in $p(x)$, this means $L(b)$ is beyond the edge of $p(x)$, so may as well be an element of $\cl(\emptyset)$, since this approaches $p(x)$ from the left.  In particular, if $y>x$ are from $p$, then they're in $\m U$ if and only if $x<R(y)$. Since $R(y)<y$ on $p(y)$, $R$ is a strictly increasing function $p\to p$.

Consider the function $g(R(x),y)$.  This is defined for any $y$ from $p$ and any $x\in (y)^-$.  Since $g$ is strictly increasing in both arguments, and $R$ is strictly increasing, the composition will be strictly increasing in both of its arguments.  If we fix $y$ (as $b$, for example), then since the function is strictly increasing in $x$, it must be a bijection from $(y)^-$ to some other non-cut $(L(y))^-$ for some $\emptyset$-definable function $L:p\to p$.  Therefore $h(x,y)=L^{-1}(g(R(x),y))$ is a function $(y)^-\to (y)^-$; it only remains to show it's nontrivial.

But if it's simple, then $L(y)=g(R(x),y)$ for all $y$ and all $x\in (y)^-$.  But this is impossible, since $g(R(x),y)$ is strictly increasing in $x$, but the equality would imply $g(R(x),y)$ is locally constant in $x$, a contradiction.  In particular, $h(x,b)$ is a nontrivial $b$-definable function from $(b)^-$ to itself, completing the proof.
\end{proof}

By repeatedly applying the previous lemma, we can produce a faithful non-cut from any unfaithful cut.

\begin{lemma}
	If $p\in S_1(T)$ is an unfaithful cut, then there is a faithful non-cut $q(x)\in S_1(\emptyset)$.
\end{lemma}
\begin{proof}
	Let $p\in S_1(T)$ be unfaithful.  We may assume $p$ is 1-nonsimple.  Fix a tuple from $p$ of minimal length which witnesses unfaithfulness; this length must be at least two, so we label it $[a]<[b]<[c_1]<\cdots<[c_k]$ where $k\geq 0$, such that for some $\emptyset$-definable $f(x,y,\o z)$, $[a]<[f(a,b,\o c)]<[b]$.  The type $q(x)=p(x)\cup\{[x]<[c_1]\}$ is a complete $\o c$-type which is $2$-unfaithful under the function $f(x,y,\o c)$, so by Lemma~\ref{firstUnfaithfulCutLemma}, there is a $b\o c$-definable function $g(x,b,\o c)$ where if $x\in (b)^-$, then $x<g(x,b, \o c)<b$.
	
	This is a definable property, so pick a $k+1$-cell $\m U$ containing the tuple $(b,c_1, \ldots, c_k)$ such that if $(y,z_1, \ldots, z_k)$ is in $\m U$, and if $x\in (y)^-$, then $x<g(x,y,\o z)<y$.  In particular, there are $\emptyset$-definable function $L(y,z_1, \ldots, z_{k-1})$ and $R(y,z_1, \ldots, z_{k-1})$ such that $L(b,c_1, \ldots, c_{k-1})<c_k<R(b,c_1,\ldots,c_{k-1})$.  By minimality of the length of the unfaithful tuple, either $R$ is above $p$ entirely or is Archimedean-equivalent to $c_{k-1}$, the latter of which would contradict the fact that $[c_{k-1}]<[c_k]$.  So $R(b,c_1,\ldots,c_{k-1})$ is above $p(x)$, so there is an element $d_k\in\cl(\emptyset)$ such that $(b,c_1, \ldots, c_{k-1},d_k)\in\m U$, and therefore, that if $x\in (b)^-$, then $x<g(x,b,c_1, \ldots, c_{k-1},d_k)<b$.
	
	In this way, we can replace all the $c_i$ with elements of $\cl(\emptyset)$.  So there is a $\emptyset$-definable $g(x,y)$ where for \emph{any} $y$ from $p$, and any $x\in (y)^-$, $x<g(x,y)<y$.  Since this property holds on an infinite set, it holds on an interval $I$, which must necessarily include all of $p$.  Since $\cl(\emptyset)$ approaches $p$ from the right, there is an element $b'\in\cl(\emptyset)$ which is in $I$.  But then the $\emptyset$-definable function $g(x,b')$ is a function from the $\emptyset$-definable non-cut $(b')^-$ to itself, completing the proof.
\end{proof}

We have now proved the following:

\begin{lemma}\label{nonsimpleNonisolatedGivesFaithful}
	If $T$ admits a nonisolated nonsimple type over $\emptyset$, then $T$ admits a faithful type over $\emptyset$.
\end{lemma}

\subsection{A Nonsimple Isolated Type}

Our goal in this subsection is to show that if there is a nonsimple, atomic type over $\emptyset$, then that type has a canonical tail. Throughout, $p$ will refer to a nonsimple atomic type, and $I$ will refer to the atomic interval which generates it.  We will refer to the left and right endpoints of $I$ as $-\infty$ and $\infty$, respectively, though they may actually be standard elements of the structure.  Throughout this section, $\cl^I(A)$ will refer to the closure of $A$ within $I$; this is used instead of $p$ to emphasize that $p$ is isolated.

Because of the restrictions in Lemma~\ref{bigNMakesClosureDense}, the cases where this type is 1-nonsimple and 1-simple are fairly different.  We deal with the 1-nonsimple case first.

\begin{lemma}
	If $p\in S_1(T)$ is a 1-nonsimple atomic type, then $p$ has a canonical tail.
\end{lemma}
\begin{proof}
	Suppose not.  Then there are $a,b,c,d$ in $I$, such that $c,d>\cl^I(ab)$ and $c\sim_a d$ but $c\not\sim_b d$.  We may assume that $c<d$; then there is a definable function $f(x,y)$ such that $f(c,a)\geq d$.  By completeness of the $c$-type $\{x\in I\}\cup\{[x]<[c]\}$, $f(c,y)$ is strictly monotone and continuous on the interval $(-\infty,c')$ for some $c'\in\cl^I(c)$.  If $f(c,y)$ is strictly increasing, then pick any $c''<c'$ in $\cl^I(c)$, observing that $a<c''$, so $f(c,c'')>f(c,a)\geq d$, so that $c\sim d$ (and therefore $c\sim_b d$ as well, a contradiction).
	
	Therefore $f(c,y)$ is strictly decreasing.  If $b<a$ then $f(c,b)>f(c,a)\geq d$ so $c\sim_b d$ again; therefore $a<b$.  Finally, if $a\sim b$, then there is $b'\leq a$ in $\cl^I(b)$, so that $f(c,b')>f(c,a)\geq d$, a contradiction.  So $[a]<[b]$, and in fact $[a]<[b]<[c]<[d]$.
	
	Let $a'=f(a,b)$, defined since $[a]<[b]$, and by construction of $f$, $[a]<[b]<[a']<[c]<[d]$.  Since $\cl(ab)=\cl(ba')$ and $c\sim_a d$, $c\sim_{ab}d$, so $c\sim_{ba'}d$.  Therefore, there is a $b$-definable function $g(x,y)$ where $g(a',c)\geq d$.  But all of the elements $a'$, $c$ and $d$ come from the non-cut $q(x)$ above $b$.  Therefore, we may say that $c\sim_{a'}d$ in $q(x)$.  Since $\cl(ba')=\cl(ab)$, by construction of $c$, $a'$ and $c$ are inequivalent in $q(x)$.  But then by faithfulness, $c\sim d$ in $q(x)$, or rather, $c\sim_b d$.
\end{proof}

We can now deal with the 1-simple case by an inductive argument, using Lemma~\ref{bigNMakesClosureDense} freely.

\begin{lemma}
	If $p\in S_1(T)$ is nonsimple and atomic, then $p$ has a canonical tail.
\end{lemma}
\begin{proof}
	Let $n$ be minimal such that $p$ is $n$-nonsimple.  By the previous lemma, we may assume $n\geq 2$.  We use the following claim as an inductive step:
	
	\begin{claim}
		Let $a,c,d$ realize $p$, and let $|A|\geq n$ be a set of realizations of $p$.  Suppose $c,d>\cl^p(Aa)$ and $c\sim_{Aa}d$. Then $c\sim_Ad$ as well.
	\end{claim}
	\begin{claimproof}
		Let $A$, $a$, $c$, and $d$ be as described.  Since $c\sim_{Aa}d$, there is an $A$-definable function $f(x,y)$ such that $f(c,a)\geq d$.  Since $c\not\sim_Ad$, $[c]<[d]$ in the non-cut $(\infty)^-_A$ (a non-cut since $n>1$); that is, the left non-cut at infinity with respect to $A$.  If $a$ also realizes this type, then because $c>\cl^I(Aa)$, it must be that $[a]<[c]<[d]$ in $(\infty)^-_A$.  Then, since $f(a,c)\in\cl^I_A(ac)$, by faithfulness of non-cuts, $f(a,c)\sim_Aa$, so $d\sim_Ac$ by convexity.
		
		First, suppose $\tp(a/A)$ is a non-cut at some $L\in\cl^I(A)$, which may be $\pm\infty$; say $a\in (L)^-$ for concreteness.  Then the formula $\lim_{y\to L^-}f(x,y)=\infty$ is satisfied by $c$, and therefore by \emph{all} $x$ which are sufficiently large over $A$.  So pick some $c'\in\cl^I(A)$ where $\lim_{x\to L^-}f(c',x)=\infty$.  Let $a'=f(c',a)$.  Then $f(c',a)\in\cl(Aa)\setminus\cl(A)$, so $\cl(Aa)=\cl(Aa')$, so $c\sim_{Aa'}d$.  Moreover, $\tp(a'/A)$ is $(\infty)^-$, so by the logic above, $c\sim_Ad$.
		
		Second, suppose $\tp(a/Ac)$ is a cut.  Then the function $f(c,y)$ must be strictly monotone at $a$, else $f(a,c)\in\cl(Ac)$, so $c\sim_Ad$.  But since $\tp(a/Ac)$ is a cut, $\cl^I(Ac)$ approaches $a$ on both sides, in particular touching the nice domain of $f(c,y)$ on both sides.  So if $f(c,y)$ is strictly increasing at $a$, then pick an $a'\in\cl(Ac)$ ``just above'' $a$, noting that $f(c,a')>f(c,a)\geq d$, so $c\sim_Ad$.  The strictly decreasing case is similar.
		
		Finally, we may assume $\tp(a/A)$ is a cut and $\tp(a/Ac)$ is a non-cut. This means there is an element $L\in\cl^I(Ac)$ where (we assume) $a\in (L)^-$, but $L\not\in\cl^I(A)$.  If $f(c,y)$ is strictly decreasing at $a$, then we may still pick an element from $\cl^I(Ac)$ ``just to the left'' of $a$ to increase $f(c,a')$; therefore, $f(c,y)$ is strictly increasing at $a$.  Therefore, there is a function $g$ over $A$ such that $g(L)=c$.  But $\tp(L/A)=\tp(a/A)$ is a cut, and $g$ sends this cut to the non-cut $(\infty)^-$, which is impossible.  This contradiction proves the claim.
	\end{claimproof}
	
	The lemma follows immediately from the claim. Let $A$ and $B$ be $n$-element sets of realizations of $p$.  Let $c$ and $d$ realize $p$ and satisfy $c,d>\cl^p(AB)$.  If $c\sim_{A}d$, then $c\sim_{AB}d$ as well.  By applying the claim $n$ times, we can remove all elements of $A$ from consideration, we get $c\sim_{B}d$, establishing the canonical tail.
\end{proof}

\subsection{A Useful Class of Linear Orders}

Our goal for this subsection is a temporary departure from model theory.  We need to produce a subclass $\b T$ of the class of countable linear orders such that $(\LO,\iso)$ is Borel reducible to $(\b T,\iso)$, and where for any $L_1,L_2\in \b T$, if $L_1$ and $L_2$ are isomorphic on a tail, then $L_1$ and $L_2$ are isomorphic.  We do this by giving two maps $f$ and $g$ from $\LO$ to itself; we define $\b T$ as the image of $g\circ f$.

We define a \emph{tail} of a linear order $L$ to be any interval of the form $[a,\infty)$, interpreted in $L$, where $a$ is in $L$.  Two orders $L_1$ and $L_2$ are \emph{tail-isomorphic}, or isomorphic on a tail, if there are tails $E_1$ of $L_1$ and $E_2$ of $L_2$ such that $E_1\iso E_2$ as linear orders.

To define the maps, first define the order $X=\{0\}\cup\{x\in\b Q:1\leq x\leq 2\}\cup\{3\}$, with the inherited order from $\b Q$.  Then, defined $f:\LO\to\LO$ by $L\mapsto L\times X$, with the lexicographic order.  That is, $f$ expands every point of $L$ to a copy of $X$.

\begin{lemma}
	For any linear orders $L_1$ and $L_2$, $L_1\cong L_2$ if and only if $f(L_1)\cong f(L_2)$.
\end{lemma}
\begin{proof}
	The left-to-right direction is obvious.  For the right-to-left direction, observe that the set $\{(x,1)\}\subset f(L)$ of ``1-points'' is uniformly definable by the formula which expresses ``there is a unique predecessor, but there is an interval to the right which is pure dense.''  Further, this is order-isomorphic to $L$ itself under the map $(x,1)\mapsto x$.  Therefore, if $f(L_1)\cong f(L_2)$, then the ``1-points'' of $f(L_1)$ are isomorphic to the ``1-points'' of $f(L_2)$, so $L_1\cong L_2$.
\end{proof}

Next, define $g:\LO\to\LO$ by $L\mapsto \omega\times \left(L\cup\{\infty\}\right)$, where $\infty$ is some point not in $L$ which is above every point in $L$.  That is, $g$ stacks up $\omega$ copies of $L$, with a separating ``$\infty$-point'' between each one; in particular each $\infty$-point has an immediate ``next'' $\infty$-point.  We will show that these $\infty$-points are (eventually) definable, even on tails.  Therefore, if $g(f(L_1))$ and $g(f(L_2))$ are isomorphic on a tail, then we can match up consecutive $\infty$-points between the tails, and get an isomorphism between $f(L_1)$ and $f(L_2)$.

\begin{lemma}
	For any linear orders $L_1$ and $L_2$, the following are equivalent:
	
	\begin{enumerate}
		\item $L_1\cong L_2$,
		\item $g(f(L_1))$ and $g(f(L_2))$ are isomorphic, and
		\item $g(f(L_1))$ and $g(f(L_2))$ are isomorphic on a tail.
	\end{enumerate}
\end{lemma}
\begin{proof}
	Two of the implications are obvious; it remains to show that if $g(f(L_1))$ and $g(f(L_2))$ are isomorphic on a tail, then $L_1\cong L_2$.  We will need a claim:
	
	\begin{claim}
		Let $L$ be any countable linear order.  There is a $\{<\}$-formula $\phi(v)$ such that for any tail $E$ of $g(f(L))$, there is a point $b\in E$ such that, $\phi(E)\cap (b,\infty)$ is exactly the set of ``$\infty$-points'' above $b$.
	\end{claim}
	\begin{claimproof}
		Let $E=[a,\infty)$ be some tail of $g(f(L))$.  Let $b$ be the first $\infty$-point satisfying $b>a$.  On $[b,\infty)$, every point which is \emph{not} an $\infty$-point has a neighborhood which is isomorphic to $f(L)$; therefore, any formula which gave its ``class'' before -- as a $0$-point, a $1$-point, a $2$-point, a $3$-point, or a ``pure dense'' point -- will still apply here.  More precisely:
		
		The pure dense points are exactly those satisfying the formula stating ``there is an open neighborhood around $v$ which is pure dense.''  The 1-points are exactly those stating ``$v$ is not pure dense, but there is a right-neighborhood which consists entirely of pure dense points,'' and the 2-points are defined symmetrically to the 1-points.  The 0-points are exactly those stating ``$v$ has an immediate successor which is a 1-point,'' and the 3-points are defined symmetrically to the 0-points.
		
		Let $c$ be any $\infty$-point above $b$.  Then every left-neighborhood $c$ contains infinitely many $0$-points, and thus is neither pure dense nor empty, so $c$ does not satisfy the defining formulas for pure dense points, 2-points, or 3-points.  If $L$ has no first element, then the right neighborhoods of $c$ will have the same properties.  Otherwise, if $L$ does have a first element, then the immediate successor of $c$ will be a 0-point.  Either way, $c$ does not satisfy the defining formulas of 0-points or 1-points.
		
		So let $\phi(v)$ be the negation of all the above defining formulas.  Then for all $x>b$, $\phi$ holds on $x$ if and only if $x$ is an $\infty$-point.  $\phi(v)$ is defined independent of everything, completing the proof.
	\end{claimproof}
	
	With this in mind, suppose $g(f(L_1))$ and $g(f(L_2))$ are isomorphic on a tail, say $[a_1,\infty)\iso [a_2,\infty)$.  Fix an isomorphism $\sigma:[a_1,\infty)\to [a_2,\infty)$.  Let $b_1\in g(f(L_1))$ and $b_2\in g(f(L_2))$ be as in the claim.  Since $\sigma$ is an order-isomorphism, it preserves $\phi$.  Let $c>\max(b_1,\sigma^{-1}(b_2))$ be some $\infty$-point, and let $c'$ be the next $\infty$-point after $c$.  Then the interval $(c,c')$ is order-isomorphic to $f(L_1)$.  Also, $\sigma(c)$ and $\sigma(c')$ are consecutive $\infty$-points in $g(f(L_2))$ by construction, so $(\sigma(c),\sigma(c'))$ is order-isomorphic to $f(L_2)$.  Since $\sigma$ is an isomorphism $(c,c')\to (\sigma(c),\sigma(c'))$, this shows $f(L_1)\iso f(L_2)$, so $L_1\cong L_2$, completing the proof.
\end{proof}

Define $\b T$ to be $\{g(f(L)):L\in \LO\}$ as a subclass of $\LO$.  Since $g\circ f$ is clearly Borel, we have shown the following:

\begin{lemma}
	$\b T$ is a subclass of $\LO$ which is Borel complete and on which isomorphism and tail isomorphism coincide.
\end{lemma}

\subsection{Nonsimple Types Prove Borel Completeness}

In this section, our goal is to show that if $T$ admits a nonsimple type, then $T$ is Borel complete.  We have already shown that if $T$ admits a nonsimple type, then $T$ admits a nonsimple type over $\emptyset$.  So we have two cases, in line with our previous work: either this type is nonisolated or atomic.  The first case is straightforward:

\begin{lemma}\label{nonisolatedGivesBC}
	If $T$ admits a nonsimple nonisolated type over the empty set, then $T$ is Borel complete.
\end{lemma}
\begin{proof}
	If $T$ admits a nonsimple, nonisolated type over $\emptyset$, then by Lemma~\ref{nonsimpleNonisolatedGivesFaithful}, $T$ also admits a faithful nonsimple type $p$ over the empty set.  Fix such a $p$.
	
	Our main concern is to show that given any countable linear order $L$, there is a countable model $\m M_L\models T$ such that $p(\m M_L)/\sim$ is isomorphic to $L$ as a linear order.  A close examination of the proof will show that this can be made a Borel function from $\LO$ to $\Mod(T)$.  Since isomorphism of models implies isomorphism of the ladders, this establishes a Borel reduction from the Borel complete relation $(\LO,\cong)$, establishing Borel completeness.
	
	So fix a countable linear order $L$, and let $X_L=\{a_\alpha:\alpha\in L\}$ be a set of realizations of $p$, such that if $\alpha<\beta$ in $L$, then $[a_\alpha]<[a_\beta]$.  Let $\m M_L$ be $\Pr(X_L)$.  Define the function $f:L\to p(\m M_L)/\sim$ by $f(\alpha)=[a_\alpha]$.  By construction of $X_L$, $f$ is injective and order-preserving.  So it only remains to show surjectivity.
	
	So let $c\in p(\m M_L)$.  Since $\m M_L$ is atomic over $X_L$, $\tp(c/X_L)$ is either algebraic or an atomic interval.  If $c\in\cl(X_L)$, then for some sequence $[a_{\alpha_1}]<\cdots<[a_{\alpha_n}]$, $c\in\cl(\o a)$.  By faithfulness, this means $c\sim a_{\alpha_i}$ for some $i$, so $[c]=f(\alpha_i)$.
	
	Alternately, suppose $\tp(c/X_L)$ is an atomic interval.  Let $(a,b)$ be an $X_L$-atomic interval in $p$ where $a<c<b$.  By faithfulness, $p$ is 1-nonsimple, so there is an $a'\in\cl^p(a)$ where $a'>a$.  Since $a\in\cl(X_L)$, we also have $a'\in\cl(X_L)$, so by $X_L$-atomicity of $(a,b)$, we have $a'\geq b$.  Clearly $a\sim a'$, so by convexity, $a\sim c$.  By the previous paragraph, $a\sim x_\alpha$ for some $\alpha$, so by transitivity, $c\sim x_\alpha$ as well, so $[c]=f(\alpha)$.
	
	Therefore $f$ is surjective, so is an isomorphism.  Thus $p(\m M_L)/\sim$ is isomorphic to $L$.
\end{proof}

\begin{lemma}
	If $T$ admits a nonsimple isolated type over the empty set, then $T$ is Borel complete.
\end{lemma}
\begin{proof}
	We will produce a Borel reduction $(\b T,\iso)$ to $(\Mod(\omega,T),\iso)$.  With this in mind, fix a nonsimple, atomic type $p$.  Let $n$ be minimal such that $p$ is $n$-nonsimple.
	
	For any linear order $L$, define an order $L^*=\{1, \ldots, n\}\cup L$, where $1<2<\cdots<n$ and $n<\alpha$ for all $\alpha\in L$.  Let $X_L=\{x_\alpha:\alpha\in L^*\}$, where for all $\alpha\in L^*$, $x_\alpha$ realizes $p$ and $x_\alpha>\cl^p(\{x_\beta:\beta<\alpha\})$.  Finally, let $\m M_L$ be $\Pr(X_L)$, a prime model of $T$ over $X_L$.  Observe that the function $L\mapsto \m M_L$ can be made Borel.
	
	For any $n$-element set $B$ from $\m M_L$, let $p_B(x)$ be the nonsimple non-cut $p(x)\cup\{x>\cl^p(B)\}$.  Our primary goal in this proof is to show that for \emph{any} $n$-element set $B$ from $p(\m M_L)$, we recover a tail of $L$ in $p_B$.  That is, $p_B(\m M_L)/\sim_B$ is isomorphic on a tail to $L$.  As before, we must divide into two cases, based on whether $p$ is 1-nonsimple, because of the restrictions in Lemma~\ref{bigNMakesClosureDense}.
	
	\begin{claim}
		If $p$ is 1-simple, then for any set $B$ from $p(\m M_L)$ with $|B|=n$, $p_B(\m M_L)/\sim_B$ is isomorphic on a tail to $L$.
	\end{claim}
	\begin{claimproof}
		Let $A=\{1, \ldots, n\}$; then by construction of $\m M_L$ and the fact that non-cuts are faithful, $p_{\o a}(\m M_L)/\sim_{A}$ is isomorphic to $L$.  It is therefore sufficient to show that for any $B$, $p_{B}(\m M_L)/\sim_{B}$ and $p_{A}(\m M_L)/\sim_{A}$ are isomorphic on a tail.
		
		So fix an $n$-element set $B$ from $p(\m M_L)$.  Since $p$ is $1$-simple, by Lemma~\ref{bigNMakesClosureDense}, $\cl^p(X_L)$ is a dense linear order without endpoints.  Since $\Pr(X_L)$ is atomic over $X_L$, $p(\m M_L)$ \emph{is} $\cl^p(X_L)$.  So by compactness, there is a finite subset $L_0\subset L^*$ containing $\{1, \ldots, n\}$ such that $AB\subset \cl^p\left(\{x_\alpha:\alpha\in L_0\}\right)$. Let $X_0$ be the tail of $X_L$ above $X_{L_0}$; that is, the set of all $x_\alpha$ such that for all $\beta\in L_0$, $\alpha>\beta$.  Since $L$ has no largest element, $X_0$ is nonempty.  We will show it forms a common tail of $p_{A}(\m M_L)/\sim_{A}$ and $p_{B}(\m M_L)/\sim_{B}$.
		
		$X_0$ forms a tail of $p_A(\m M_L)/\sim_A$ under the function $x_\alpha\mapsto [x_\alpha]$, by the characterization of $p_A(\m M_L)/\sim_A$ at the beginning of this proof.  As for $p_B$, by construction of $L_0$, if $x_\alpha\in X_0$, then $x_\alpha$ realizes $p_B$.  Each of the $x_\alpha\in X_0$ is $\sim_{L_0}$-inequivalent by construction of $X_L$, so must be $\sim_B$-inequivalent as well; it only remains to show that the set $\{[x]:x\in X_0\}$ is right-closed in $p_B(\m M_L)/\sim_B$.
		
		So suppose $x_\alpha\in X_0$ and $c>x_\alpha$ realizes $p$.  By the characterization of $p_A(\m M_L)/\sim_A$, $c\sim_A x_\beta$ for some $\beta\geq\alpha$.  Since $c$ and $x_\beta$ are both greater than or equal to $x_\alpha$, which is above $\cl^p(\{x_\gamma:\gamma<\alpha\})\supset\cl^p(AB)$, we can use the canonical tail condition to conclude that $c\sim_B x_\beta$ as well.  Therefore, $[c]\in\{[x]:x\in X_0\}$, so $X_0$ forms a tail of $p_B(\m M_L)/\sim_B$.  This completes the proof.
	\end{claimproof}
	
	\begin{claim}
		If $p$ is 1-nonsimple, then for any set $B$ from $p(\m M_L)$ with $|B|=n$, $p_B(\m M_L)/\sim_B$ is isomorphic on a tail to $L$.
	\end{claim}
	\begin{claimproof}
		Let $a=x_1$. As before, we can conclude that $p_a(\m M_L)/\sim_a$ is isomorphic to $L$, and therefore that we need to show for every $b\in p(\m M_L)$, $p_b(\m M_L)/\sim_b$.  So, fix such a $b$.  Since $\tp(b/X_L)$ is atomic, either $b\in\cl(X_L)$, or $\tp(b/X_L)$ is generated by an atomic interval.
		
		If $b\in\cl(X_L)$, then the previous proof applies without change.  Therefore, assume $\tp(b/X_L)$ is an atomic interval $(L,R)$ where $L,R\in\cl(X_L)$.  Let $L_0$ be a finite subset of $X_L$ which contains $1$ and such that $L,R\in\cl(\{x_\alpha:\alpha\in L_0\})$.  Let $X_0$ be the elements of $X_L$ which are above $L_0$.  This is a right-closed subset of $L$, so it forms a tail of $p_a(\m M_L)/\sim_a$; it remains to show it forms a tail of $p_b(\m M_L)/\sim_b$.  As before, the function $x\mapsto [x]$ is a well-defined, order-preserving injection.  It remains to show surjectivity.
		
		So pick a $c$ from $p_b(\m M_L)$ such that for some $x_\alpha\in X_0$, $c>x_\alpha$.  For some $\beta\geq\alpha$, $c\sim_a x_\beta$; we want to show $c\sim_b x_\beta$ as well.  Since $p$ has a canocical tail, it is enough to show that $x_\alpha>\cl^p(ab)$, so \emph{suppose not}.  Then there is an $a$-definable function $f(x)$ where $f(b)\geq x_\alpha$.  Then $f(x)$ is defined and strictly monotone on the atomic interval $(L,R)$; we may assume strict increasing.  Since $x_\alpha>\cl^p(X_0)$, it must be that $\lim_{x\to R^-}f(x)=\infty$.  We will use this limit to prove that $(L,R)$ is not $X_0$-atomic, yielding a contradiction.
		
		The image of $(L,R)$ under the function $f$ must also be an interval, since $f$ is continuous and strictly increasing, and so by the argument above, it must be of the form $(C,\infty)$ for some $C\in\cl^p(X_0)$.  By 1-nonsimplicity, there is a $C'>C$ in $p$ which is $C$-definable; since the interval is right-infinite, $C'\in\im(f)$.  But then $f^{-1}(C')\in (L,R)$ and is $X_0$-definable, a contradiction of atomicity of $(L,R)$.
	\end{claimproof}
	
	Having performed these two claims, the result follows immediately.  Let $L_1,L_2\in\b T$; we want to show $L_1\cong L_2$ if and only if $\m M_{L_1}\iso\m M_{L_2}$.  The left-to-right direction is obvious, so suppose $\m M_{L_1}\iso \m M_{L_2}$ and let $\sigma:M_{L_1}\to \m M_{L_2}$ be an isomorphism.
	
	Let $A=\{x_1, \ldots, x_n\}\subset p(\m M_{L_1})$ be the ``intended'' set of parameters for $\m M_{L_1}$.  Then $\sigma$ induces an order isomorphism from $p_A(\m M_{L_1})/\sim_A$ to $p_{\sigma(A)}(\m M_{L_2})/\sim_{\sigma(A)}$.  The former is isomorphic to $L_1$, and the latter is isomorphic (on a tail) to $L_2$.  Thus $L_1$ is isomorphic to a tail of $L_2$.  Since both orders are from $\b T$, this shows $L_1\iso L_2$, completing the proof.
\end{proof}

Combining this with results from the above, we have proved the main theorem of the section.

\begin{thm}
	Let $T$ be a countable o-minimal theory.  If $T$ admits a nonsimple type, then $T$ is Borel complete.
\end{thm}

\section{Corollaries}\label{CorollarySection}

Most interesting o-minimal theories admit nonsimple types, and are therefore Borel complete.  Our aim for this section is to establish two broad classes of o-minimal theories which are Borel complete; these yield very general sufficient conditions for Borel completeness.

The first such class of such theories is the class of nontrivial theories - those where it is possible for a point to be definable over a set without being definable over any single point inside that set.  For example, any theory with an infinite definable group would satisfy this property.  We show that any nontrivial o-minimal theory is Borel complete by using nontriviality to construct a nonsimple type over a finite set, then appealing to Theorem~\ref{MainTheorem}.

The other broad class is the discretely o-minimal theories, or even those which have a significant discrete part.  Although it was shown in \cite{discreteOMinimal} that the discrete part of an o-minimal theory is completely trivial (in the above sense), the successor function still provides an interesting (unary) function on the structure, which is enough to construct a nonsimple type and show Borel completeness.

\subsection{Nontrivial Theories}\label{TrivialitySection}

Recall that a theory $T$ \emph{nontrivial} if there is some point $b$ and some set $A$ where $b\in\cl(A)$ but $b\not\in\bigcup_{a\in A}\cl(a)$.  We use exchange and nontriviality to produce a nonsimple type over finitely many parameters, and therefore conclude with Borel completeness.

\begin{thm}
	If $T$ is a nontrivial o-minimal theory then $T$ is Borel complete.
\end{thm}
\begin{proof}
	Suppose $T$ is nontrivial.  We will produce a nonsimple type $p(x)$ over finitely many parameters, establishing Borel completeness by Theorem~\ref{MainTheorem}.  By nontriviality, there is a set $A$ and $b\in\cl(A)$ where $b\not\in\cl(a)$ for any $a\in A$.  We may assume $A$ is finite, and that $A$ has minimal cardinality among all ``nontrivial sets.''  Enumerate $A$ in an ascending way as $a_1<\cdots<a_n$, remarking that $n\geq 2$.  The set $B=\{a_3, \ldots, a_n\}$ will be the first part of our parameter set.
	
	Then $b\in\cl_B(a_1, a_2)$ but $b\not\in\cl_B(a_i)$ for either $i$.  Let $p(x)=\tp(a_1/B)$, $q(x)=\tp(a_2/B)$, and $r(x)=\tp(b/B)$.  Each of these types is nonalgebraic.  Suppose (for example) that $\cl^p_B(a_2)$ is nonempty; then we may replace $a_2$ with some realization of $p(x)$, bidefinable with $a_2$ over $B$, without affecting the dependence relation $b\in\cl(a_1a_2)\setminus(\cl(a_1)\cup\cl(a_2))$.  Using this idea and exchange, we may assume the following cases are exhaustive:
	
	{\bf First:} One of the types $p$, $q$, or $r$ is a nonsimple $B$-type, in which case the theorem is proved.
	
	{\bf Second:} $p=q=r$; then $p$ is a 2-nonsimple $B$-type under whatever function takes the pair $(a_1,a_2)$ to $b$.
	
	{\bf Third:} $p=q$ and $p\not=r$.  Then there is a $B$-definable binary function $f:p^2\to r$ taking ascending pairs from $p$ into single elements of $r$.  Then for any $a$ modeling $p$, there is a unique extension of $r$ to a $Ba$-type (or else we're actually in the previous case) and the function $f(a,y)$ must be a bijection from the complete $Ba$-type $p(x)\cup\{x>a\}$ to $r$.  But then the function $g:p\to p$ where $g(x)$ is the unique $y>x$ such that $f(x,y)=b$ is well-defined and nonsimple, so that $p(x)$ is a complete, nonsimple $Bb$-type.
	
	{\bf Fourth:} $p$, $q$, and $r$ are all distinct.  Let $f$ be such that $f(a_1,a_2)=b$.  Then for any $c$ modeling $r$, the types $p$ and $q$ are completely described over $Bc$ and the function $f(a,y)$ is a bijection from $q$ to $r$.  So for any $c$ realizing $r$, we have a bijection $h_c:p\to q$ taking $x$ to the unique $y$ where $f(x,y)=c$.
	
	Therefore, fix $c_1<c_2$ realizing $r$.  If $p$ or $q$ does \emph{not} extend uniquely to a complete $Ac_1c_2$-type, then we have a function $r^2\to p$ or $r^2\to q$, and $T$ is Borel complete by a previous case.  But otherwise, $p$ and $q$ are complete over $Ac_1c_2$, and the functions $h_{c_1}$ and $h_{c_2}$ are \emph{distinct} bijections $p\to q$.  Therefore $h_{c_1}^{-1}\circ h_{c_2}$ is a nontrivial bijection $p\to p$, so $p$ is a nonsimple type over $Bc_1c_2$.
\end{proof}

Therefore $T$ is Borel complete.

\subsection{Non-Dense Theories}\label{DiscreteSection}

Given an o-minimal theory $T$, a model $\m M\models T$, say a point $a\in\m M$ is \emph{non-dense} if $a$ has either an immediate successor or an immediate predecessor (which may be among $\pm\infty$).  If $T$ has only finitely many such points, they play no role in the countable model theory of $T$; we can canonically fit a copy of $(\b Q,<)$ between any non-dense point and its successor or predecessor.  Our theorem for this section is the following:

\begin{thm}\label{DiscreteBC}
	If $T$ is an o-minimal theory with infinitely many non-dense points, then $(\Mod(T),\iso)$ is Borel complete.
\end{thm}

\begin{proof}
	We construct a nonsimple type over the empty set.  Since there are infinitely many non-dense points, there is an infinite interval $I_0$ over $\emptyset$ which consists entirely of non-dense points.  Therefore, there is a subinterval $I$ of $I_0$ of points which all have immediate successors and predecessors.  Let $S(x)$ denote the immediate successor function, where it is defined. We will construct a complete type extending $I$ which is nonsimple under the function $x\mapsto S(x)$.
	
	Let $I=(a,b)$, noting $a,b\in\cl(\emptyset)$.  We have several cases:
	
	{\bf First:} If $a$ has no immediate successor, then define $p(x)$ by $$p(x)=\{a<x\}\cup\{x<c:c>a, c\in\cl(\emptyset)\}$$  By o-minimality, $p(x)$ is a complete type, and clearly extends $I$.  It may be either an atomic interval (if $\cl^I(\emptyset)=\emptyset$) or a non-cut $(a)^+$ (if not), but either way, it must be closed under $S$.  For if not, there is an $x$ realizing $p(x)$ such that $S(x)\geq c$ for some $c\in\cl^I(\emptyset)$.  But then $S(x)=c$, so $S^{-1}(c)$ is well-defined, in $\cl(\emptyset)$, and equal to $x$, so that $x$ does not model $p$ after all.
	
	Thus $p(x)$ is complete and nonsimple under the function $S$.
	
	{\bf Second:} If $b$ has no immediate predecessor, then define $q(x)$ by $$q(x)=\{x<b\}\cup\{x>c:c<b, c\in\cl(\emptyset)\}$$  By the same logic as above, $q(x)$ is complete and closed under the function $x\mapsto S^{-1}(x)$, so is nonsimple.
	
	{\bf Finally:} If $S(a)$ and $S^{-1}(b)$ both exist, then define $r(x)$ by $$r(x)=\{x>S^n(a):n\in\omega\}\cup\{x<c:c\in\cl(\emptyset)\land c>S^n(a)\textrm{ for all }n\in\omega\}$$ Then $r(x)$ is a complete type over $\emptyset$ as before, and is a cut.  But as before, if $x$ models $r$, then $S(x)$ is defined and must still realize $r$.  Thus $r$ is nonsimple.
\end{proof}

Therefore $T$ is Borel complete.

\bibliography{Citations}
\end{document}